\documentclass[letterpaper, 10 pt, conference]{ieeeconf}  
\IEEEoverridecommandlockouts                          
\title{\LARGE \bf
A Mean-Field Game Model For Large-Scale Attrition in Attacker--Defender Systems
}

\usepackage[utf8]{inputenc}
\usepackage{amsfonts}
\usepackage{hyperref}
\usepackage{amsmath}
\usepackage{xcolor}
\DeclareUnicodeCharacter{2212}{-}

\usepackage{amsthm}

\usepackage{pdflscape}
\usepackage{mathrsfs}
\usepackage{bigints}
\usepackage{graphicx}
\usepackage{subcaption}

\hypersetup{
colorlinks=true,
linkcolor=red,
urlcolor=blue,
citecolor=red
}

\usepackage{bm}
\usepackage{commath}
\usepackage{amssymb}
\usepackage{eurosym}
\usepackage{float}
\usepackage{graphics}
\usepackage{epsfig}
\usepackage{placeins}

\allowdisplaybreaks
\makeindex

\renewcommand{\div}{\operatorname{div}}

\newtheorem{theorem}{Theorem}

\newtheorem{problem}{Problem}

\newcommand{\T}{\mathbb{T}}

\author{$^*$Avetik Arakelyan$^{3,1}$, $^*$Tigran Bakaryan$^{3,2,1}$, Davit Alaverdyan$^{2,1}$, Naira Hovakimyan$^{4}$ and Isaac Kaminer$^{5}$%
\thanks{$^{*}$These authors contributed equally.
}%
\thanks{$^{1}$Yerevan State University, Armenia; emails: 
\texttt{\{avetik.arakelyan, davit.alaverdyan2\}@ysu.am}}%
\thanks{$^{2}$Center for Scientific Innovation and Education, Armenia}%
\thanks{$^{3}$Institute of Mathematics of NAS RA, Armenia; email: \texttt{tigran.bakaryan@instmath.sci.am}}%
\thanks{$^{4}$Department of Mechanical Science \& Engineering, University of Illinois Urbana-Champaign, USA; email: \texttt{nhovakim@illinois.edu}}%
\thanks{$^{5}$Department of Mechanical \& Aerospace Engineering, Naval Postgraduate School, Monterey, CA, USA; email: \texttt{kaminer@nps.edu}}%
\thanks{The research is supported by the Higher Education and Science Committee of the MESCS RA under Research Project No.~24IRF-1A001 and in part by the Air Force Office of Scientific Research Grant (AFOSR) Grant AF FA9550-25-1-0274, the National Aeronautics and Space Administration (NASA) under Grant 80NSSC22M0070, and by the National Science Foundation (NSF) under Grants CMMI 2135925, CPS 2311085 and IIS 2331878.
\#80NSSC22M0070.}%
}

\begin{document}

\maketitle
\thispagestyle{empty}
\pagestyle{empty}

\begin{abstract}
This paper proposes a novel Mean-Field Game (MFG) framework for large-scale attacker–defender systems aimed at protecting one or multiple High-Value Units (HVUs). Motivated by classical agent-wise attrition models, we introduce a population-wise attrition mechanism formulated by statistical distance between populations, enabling a macroscopic description of weapon-based interactions between large populations. Leveraging this and  Lions derivative on the space of probability measures, we derive the associated MFG system, which characterizes optimal strategies and the evolution of population distributions in attacker–defender interactions.
We analyze the model by establishing upper and lower bounds on the defender density, ensuring physical consistency by preventing concentration and depletion. For numerical investigation, we develop a numerical scheme combining physics-informed neural networks with Sinkhorn method to solve attacker-defender MFG system.  Simulations confirm the effectiveness of the framework and reveal key insights, including sensitivity to weapon strengths and population dispersion. 
\end{abstract}

\section{Introduction}

Recent technological advances have made robotic swarms a practical
reality, including in military operations such as swarm attacks. This
creates a growing need for autonomous protection of critical
infrastructure against large-scale adversarial threats, where
decision- and game-theoretic methods play a central role
\cite{hunt2024review}. In this work, we study this problem in a
large-scale setting with many attackers, defenders, and HVUs, and
propose a population-wise formulation leading to an attrition-based MFG
attacker--defender model.

Traditionally, adversarial engagements are modeled using
Lanchester-type combat models
\cite{protopopescu1989combat,lanchester1916aircraft}, in which the
states of the agents and the engagement dynamics are combined at a
macroscopic level. However, such models do not explicitly capture weapons impact. More recent works such as
\cite{walton2022defense,tsatsanifos2021,kang2025online},
motivated by Koopman-type radar models, propose agent-wise attrition
formulations in which the weapons of attackers and defenders are modeled
through attrition rate functions, while the two sides act with opposing
objectives: attackers seek to destroy the HVU (or HVUs), whereas
defenders seek to protect it. While these agent-level models provide detailed and
high-fidelity descriptions for small-scale engagements, they become
computationally prohibitive as the numbers of attackers, defenders and HVUs
grow.
To address this scalability issue, we adopt the mean-field game (MFG)
framework. MFGs provide an effective mathematical framework for studying
strategic decision-making in large populations of interacting agents;
see, for instance, \cite{liu2018mean,elamvazhuthi2020mean}. Nevertheless,
MFG formulations for combat-type attacker--defender interactions remain
relatively limited. A related example is \cite{wang2022multi}, where the
interaction is described through attraction--repulsion forces. 
\newline
\noindent \hspace*{7pt} 
In this paper, we propose a novel population-wise attacker--defender MFG model that bridges classical mean-field game formulations and combat attrition models. Building on the survival-probability and damage
attrition concepts in \cite{kang2025online}, we lift these interactions
from the level of individual agents to the level of populations by
introducing a statistical distance between their distributions. This
enables us to model the lethality and spatial interaction between large
attacker and defender swarms. In particular, we employ the
Wasserstein-$2$ distance, which captures an important tactical feature:
the effectiveness of a defender swarm depends not only on its proximity
to the attacker swarm, but also on its spatial spread.
\newline
 \noindent \hspace*{7pt}
 The main contributions of this paper are threefold. First, we introduce
a population-wise attrition rate function (see Sections~\ref{section 3}). Relaying on this, we derive an
attacker--defender mean-field game model given by a coupled backward
Hamilton--Jacobi--Bellman (HJB) and forward Fokker--Planck (FP) equations, \ref{sec:MFG-deriv}. Second, we
establish rigorous two-sided bounds for the defender density, ensuring
physically meaningful evolution by preventing unrealistic concentration
or depletion; see Section~\ref{sec:flowb}. Third, we develop a numerical
scheme based on physics-informed neural networks (PINNs), combined with a
well-known Sinkhorn algorithm for approximation of optimal transport, to solve the resulting
MFG system; see Section~\ref{sec:sim}. The numerical results highlight
strategically relevant effects, including the sensitivity of HVU
protection to weapon strength and to the initial spatial dispersion of
the defender population.

\section{Preliminaries}\label{sec:prelim}

To motivate our population-wise attrition based model,  in this section, we recall agent-wise attrition function based attacker-defender model considered in \cite{walton2018optimal,tsatsanifos2021,walton2022defense, kang2025online}.
In this model the authors treat the problem of optimal defense of a high-value unit (HVU) against a large-scale adversary attack. In the
proposed problem setting, the attackers attempt to destroy
the HVU by approaching it and firing weapons, while
the defenders attempt to save the HVU by shooting down
the attackers.

For the simplicity, we assume that attackers are kamikaze, i.e. they ignore defenders and have single
objective, the destruction of the HVU. Furthermore, we suppose that the dynamics/behavior of attackers is known.

Based on these assumptions, we provide a mathematical model. The general dynamics of attackers and defenders is given by
\[
\begin{cases}
\dot{s}_{i}=F(s_{i},\dot{s}_{i},s_{-i},\dot{s}_{-i},s_{HVU}),\\
\dot{x}_{k}=u_{k},   
\end{cases}
\]
where $k=1,\ldots,M$ and $i=1,\ldots,N.$ Here,  $s_i$ and $x_k$  stand for positions of an $i$-th attacker and $k$-th defender, respectively. The HVU is assumed to be unique and  remain in a fixed position during the time.
 
 To model the attrition of attackers and HVU, the authors in \cite{walton2018optimal,walton2022defense} introduce attrition rate functions that model the weapon of an agent (attacker or defender), which takes the relative distance between the agent and a target as an argument and returns the rate of destruction of the target. Particularly, the rate at which the attacker \(i\) is shot down by defender \(k\) is given by
$$
d^{att}_{ik}(s_{i},x_{k})=d^{att}_{ik}\left(\|s_{i}-x_{k}\|\right).
$$
In the same way, the attrition rate at which HVU is destroyed by attacker \(i\) is given by
$$
d^{H}_{i}(s_{i},s_{HVU})=d^{H}_{i}\left(\|s_{i}-s_{HVU}\|\right),
$$ 
where \(s_{HVU}\) is the position of HVU. 
As an example of attrition rate function (see \cite{kang2025online}) one can take the inverted cumulative normal distribution. It is easy to see that when the distance is \(0\), this function takes a value of \(1\). Otherwise, it decreases smoothly toward \(0\) as the distance increases.  This reflects the core dynamics where the target's chance of survival decreases as the agent gets closer or employs stronger weapons.

The concept of attrition described above allows the attacker–defender interaction to be modeled in a decoupled manner from the perspective of individual agents (attackers or defenders), while coupling arises only through shared targets or collective mission objectives.

Using the definition of the attrition function, we now define the survival probability of attackers.
Specifically, the survival probability of attacker \(i\) at the location $s_i(t)$ from defender \(k\) at the location $x_k(t)$ over a short time interval \([t,t+\Delta t]\) is given by  
\[
1-d^{att}_{ik}(s_{i}(t),x_{k}(t))\Delta t.
\]
Then, denoting by  \(Q_{ik}(t)\)   the survival probability of attacker \(i\) from defender \(k\) at time $t$, we have
\[
Q_{ik}(t+\Delta t)=Q_{ik}(t)\left(1-d^{att}_{ik}(s_{i}(t),x_{k}(t))\Delta t \right).
\] 
This leads to 
\begin{equation}
    \label{def-ode-Q}
\dot{Q}_{ik}(t) =-Q_{ik}(t)d^{att}_{tk}(s_{i}(t),x_{k}(t)), \;\;\text{as}\;\;\Delta t\to 0.
\end{equation}
 For the attacker $i$ to remain alive at time $t$, it must survive the attacks of all defenders throughout the entire interval $[0,t]$. Therefore, letting $Q_{i}(t)$ denote the overall survival probability of attacker $i$, we have
\[
Q_{i}(t)=\Pi^{M}_{k=1}Q_{ik}(t).
\]
 
Similarly, for the HVU to survive, it must withstand all attacks from the attackers.
Hence, the survival probability of the HVU over the short interval \([t,t+\Delta t]\) is the product of its survival probabilities against each attacker $i$, and can therefore be expressed as
\(\Pi^{N}_{i=1}\left(1-d^{H}_{i}Q_{ik}(t)\Delta t\right)\). 
Denoting \(P(t)\) as the survival probability of HVU from all the attackers, we arrive at
\[P(t+\Delta t)=P(t)\prod_{i=1}^{N}\left(1-d^{H}_{i}(s_{i}(t),s_{HVU}(t)) Q_{i}(t)\Delta t\right),\] 
which leads to
\begin{equation}\label{def-ode-P}
    \dot{P}(t) =-P(t)\sum_{i=1}^{N}d^{H}_{i}(s_{i}(t),s_{HVU}(t))Q_{i}(t), 
\end{equation}
as $\Delta t\to 0.$
At the beginning of the engagement ($t=0$), all attackers and the HVU are assumed to be alive. Therefore, the initial conditions are set as \(Q_{ik}(0)=1\) and \(P(0)=1\). With these initial conditions,  the explicit solutions of the ODEs in \eqref{def-ode-Q} and \eqref{def-ode-P} are given by
\begin{align*}
Q_{ik}(t)& =e^{-\int_{0}^{t}d^{att}_{tk}(s_{i}(\tau),x_{k}(\tau))d\tau},\\ 
P(t)&=e^{-\int_{0}^{t}\sum_{i=1}^{N}d^{H}_{i}(s_{i}(\tau),s_{HVU }(\tau))Q_{i}(\tau)d\tau} \\
&=e^{-\int_{0}^{t}\sum_{i=1}^{N}d^{H}_{i}(s_{i}(\tau),s_{HVU }(\tau))\Pi^{M}_{k=1}Q_{ik}(\tau)d\tau}.
\end{align*}

Thus,  the survival probability of the HVU at the terminal time \(t=t_{f}=T\) is
\begin{multline*} 
        J(u_{1},\ldots,u_{M}) = 1-P(T)\\=1-e^{-\int_{0}^{t_{f}}\sum_{i=1}^{N}d^{H}_{i}(s_{i}(\tau ),s_{HVU}(\tau))\Pi^{M}_{k=1}Q_{ik}(\tau)d\tau},
\end{multline*}
where $u_{i}$ is the control input of the defender (decision maker) $i$.
Having established this formulation, the objective-maximization of the HVU’s survival probability-can be posed as a distributed optimization problem.
However, to enable scalability to scenarios involving a large number of agents, we state the problem in the form of an equivalent game formulation.
This formulation also facilitates further generalization, allowing for heterogeneous agent objectives in more complex settings. 
\begin{problem}\label{problem-1-Isaac}
Find the optimal defensive control inputs \(u_{k}\), \(k=1,\ldots, M\),  solving the following optimization problems:
   \begin{align*}
\min_{u_{k}} & \; J(u_{1},\ldots,u_{M})=\min_{u_{k}}(1-P(T)) \\
\text{s.t.} & \quad \dot{s}_{i}=F(s_{i},\dot{s}_{i},s_{-i},\dot{s}_{-i},s_{k},s_{HVU}) \\
& \quad \dot{x}_{k}=u_{k}.
\end{align*} 
\end{problem}
\section{Population-wise Setting}\label{section 3}
In this section, we introduce a population-wise attrition rate function
that captures macroscopic interactions at the level of distributions, and
derive the corresponding HVU protection model governed by this mechanism.
 
\subsection{Population-wise Attrition Rate Function}\label{section 3.1}

In this subsection, we introduce a population-wise attrition rate function
defined at the level of shooter and target populations. Each population
is represented by a distribution, and the attrition of the target population
is determined by a statistical distance between these distributions.

Suppose that we are given populations of shooters and targets with
distributions $\sigma_s$ and $\sigma_t$, respectively. Let
$R(\sigma_s,\sigma_t)$ denote a statistical distance between these
distributions, to be specified later. We define the population-wise
attrition rate as
\begin{equation*}
    d^{att}(\sigma_s, \sigma_{t})\equiv d^{att}(R(\sigma_s, \sigma_{t})).
\end{equation*}
The attrition rate depends only on the statistical distance between the shooter and target populations.
This formulation reflects the idea that
the effectiveness of the attack is primarily determined by the relative
spatial configuration of the two populations. In particular, we assume
that shooters are equipped with weapons whose effectiveness depends on
the proximity between shooters and targets. Throughout this paper, we
focus on such interaction mechanisms; other weapon models can be handled
in a similar manner.

As a specific example, we may consider the function
\[
d^{att}\left(R\left(\sigma_s, \sigma_{t}\right)\right)=\lambda \exp(-\tfrac{R^2(\sigma_s, \sigma_{t})}{\sigma}).
\]
Note that this attrition rate attains its maximum value $\lambda$ when
the statistical distance between shooters and targets is zero, and
decreases to $0$ as the distance increases.
 The parameters $\lambda > 0 $ and $\sigma  >0$ are fixed constants that determine the overall strength of the attrition and its sensitivity to the distance between the populations.

In contrast to agent-wise attrition models (see, e.g., \cite{tsatsanifos2021}),
the proposed population-wise formulation is inherently averaged and
scalable. This enables modeling of interactions between large
populations of shooters and targets in a tractable manner.

The remaining modeling choice is the selection of the statistical
distance $R(\cdot,\cdot)$. In Fig.~\ref{fig:distances}, we compare several
statistical distances and illustrate their behavior in two scenarios:
 when two populations with Gaussian distributions (with fixed variance)
move closer to each other by varying their centers of mass and when the centers of mass are fixed and the variance of one population
is varied. 
In the first scenario, one Gaussian distribution with variance
$\sigma^2 = 0.85$ moves from $(2,2)$ toward $(0,0)$, while the second
Gaussian remains fixed at $(0,0)$ with variance $\sigma^2 = 0.85$.
In the second scenario, one Gaussian is centered at $(0,0)$ and its
variance increases from $\sigma^2 = 0.1$ to $\sigma^2 = 10$, while the
second Gaussian remains fixed at $(2,2)$ with variance $\sigma^2 = 1.5$.
The first scenario highlights how rapidly each distance decreases as the
centers of mass approach each other, while the second illustrates the
sensitivity of each distance to changes in the shape of the populations.
\begin{figure*}[t]
\centering
\makebox[\textwidth][c]{%
    \hspace{-0.cm}
    \includegraphics[scale=0.3]{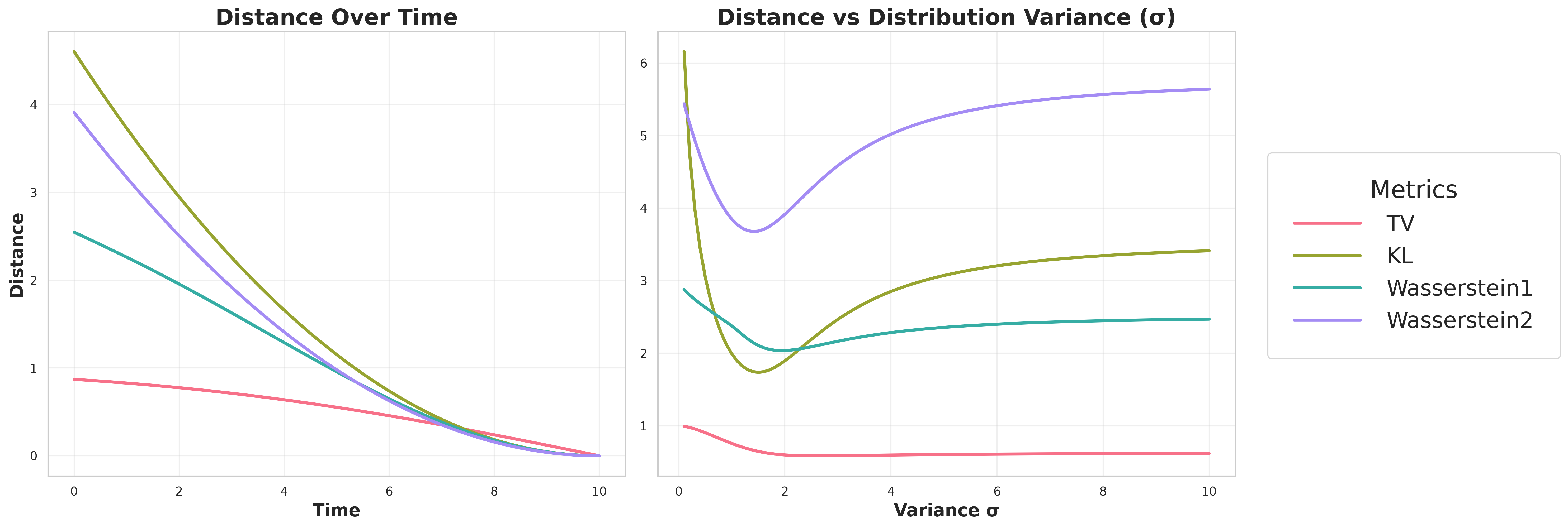}
}
\caption{Comparison of statistical distances.
Left: moving distributions.
Right: varying variance.
}
\label{fig:distances}
\end{figure*}

The intuition suggests that, in the first scenario, the distance should
decrease as the populations approach each other; this behavior is captured
by all considered distances (see the left plot of Fig.~\ref{fig:distances}).
In the second scenario, the situation is more subtle. If the shooter
population is highly concentrated, it may not effectively cover the
target population. When moderately spread, its effectiveness improves,
while excessive dispersion again reduces its effectiveness due to
dilution of its impact.
 This nontrivial behavior is effectively captured by the Wasserstein-2
distance and the Kullback--Leibler divergence (see the right plot of Fig.~\ref{fig:distances}). Motivated by this, and by
the fact that the Wasserstein-2 distance defines a true metric on the
space of probability distributions, we focus our study on the
Wasserstein-2 distance.

\subsection{Optimization Problem}\label{sec:OP}
Following the agent-wise model discussed in Section~\ref{sec:prelim},
we derive the corresponding optimization formulation for the HVU
protection problem at the population level.

In the spirit of the population-wise attrition framework and mean-field
game theory, we consider populations of defenders, attackers, and HVUs,
described through their representative agents with states $S_t$, $X_t$,
and $Y_t$ at time $t$, respectively.
For simplicity of the subsequent analysis, we assume that
$X_t, S_t, Y_t \in \mathbb{T}^d$.
The corresponding distributions are denoted by $m(t)$ and $\mu(t)$ for
the defender and attacker populations, i.e., $S_t \sim m(t)$ and
$X_t \sim \mu(t)$, while the HVU population is static, i.e.,
$Y_t = Y \sim \nu_{\mathrm{H}}$.

The dynamics of the attacker and defender populations are governed by the
following system of stochastic differential equations:
\begin{equation}\label{McKean-Vlasov}
\begin{cases}
\mathrm{d}X_t = u_{\mathrm{att}}\big(t, X_t\big)\,\mathrm{d}t + \sqrt{2\varepsilon}\,\mathrm{d}W_t^X, \\
\mathrm{d}S_t = u\big(t, S_t\big)\,\mathrm{d}t + \sqrt{2\varepsilon}\,\mathrm{d}W_t^S,
\end{cases}
\end{equation}
where $W_t^X$ and $W_t^S$ are independent standard Brownian motions.
We assume that the behavior of the attackers is given, i.e.,
$u_{\mathrm{att}}$ is known, and that attackers pursue a single objective, namely, the destruction
of the HVU population. The defender control $u$ is to be
determined so as to maximize the survival probability of the HVU population.

Let $Q(t)$ denote the survival probability of the attacker population
over the time interval $[0,t]$. Following the classical formulation, we
assume that
\[
Q(t+\Delta t) = Q(t)\big(1 - d^{\mathrm{att}}(\mu(t), m(t))\,\Delta t\big),
\]
where $d^{\mathrm{att}}$ is the population-wise attrition rate of the
attackers, determined by the distributions of the defender and attacker
populations, $m(t)$ and $\mu(t)$, respectively.
Passing to the limit as $\Delta t \to 0$, we obtain
$\dot{Q}(t) = -Q(t)\, d^{\mathrm{att}}(\mu(t), m(t))$,
which yields
\begin{equation}
    \label{eq-Q-exp}
Q(t) = \exp\!\left(-\int_0^t d^{\mathrm{att}}(\mu(\tau), m(\tau))\,d\tau\right).
\end{equation}
Similarly, let $P(t)$ denote the survival probability of the HVU population
over the time interval $[0,t]$, and let $d^{H}$ denote the corresponding
attrition rate. Then,
\[
\dot{P}(t) = -P(t)\, d^{H}(\nu_{\mathrm{H}}, \mu(t))\, Q(t),
\]
where $\nu_{\mathrm{H}}$ stands for the distribution of HVU.
This implies
\begin{equation}\label{eq:-P-exp}
    P(t) = \exp\!\left(-\int_0^t d^{H}(\nu_{\mathrm{H}}, \mu(\tau))\,
Q(\tau)\, d\tau\right).
\end{equation}
Recalling the explicit form of $Q(t)$ in \eqref{eq-Q-exp}, we obtain
$$
P(t) = e^{-\bigints_0^t d^{H}(\nu_{H}(\tau), \mu(\tau))\cdot e^{-\bigintsss_0^\tau d^{att}(\mu(\theta), m(\theta))d\theta}d\tau}.
$$
Thus, the defender population seeks to minimize
\begin{equation}\label{main_minimizer}
\min_{u} \; \mathbb E\!\left[\big(1 - P(T)\big)\right].
\end{equation}
Since the exponential function is monotone, the above problem is
equivalent to
\begin{equation}
\begin{split}
   & \min_{u} \; \mathbb E\!\left[\big(1 - P(T)\big)\right]
= \max_{u} \; \mathbb E\!\left[P(T)\right]
\\&= \max_{u} \;\mathbb E\!\left[ \log P(T)\right]
= \min_{u} \; \mathbb E\!\left[\big(-\log P(T)\big)\right].
\end{split}
\end{equation}
In addition, we penalize the control effort. For a given parameter
$\alpha \in (0,1)$, we consider the following weighted objective:
$\min_{u} \; \mathbb E\!\left[
(1-\alpha)\big(-\log P(T)\big)
+ \alpha \int_0^{T} u^2 \, d\tau
\right]$.
Using this formulation together with \eqref{eq:-P-exp}, we obtain the
following optimization problem:
find control $u$ minimizing
\begin{equation}\label{minimizer_revised-v}
 \mathbb E\!\left[\int_0^{T}(1-\alpha)d^{H}_\mu(\tau) Q(\tau)+\alpha u^2d\tau\right],
\end{equation}
subject to \eqref{McKean-Vlasov}.

This is a mean-field (McKean--Vlasov) control problem, where the cost
depends on the distribution $m(t,\cdot)$. In the standard dynamic
programming formulation, the value function depends on both the state
and the law, $U(t,x,m)$, leading to an infinite-dimensional
Hamilton--Jacobi equation (the master equation) on
$[0,T]\times\mathbb{T}^d\times\mathcal{P}_2(\mathbb{T}^d)$.
This formulation is analytically delicate and is typically computationally
intractable. Therefore, in the next section, we adopt a representative-agent
reduction, which yields a mean-field game (MFG) system that is both
analytically and numerically tractable.

\vspace{-0.3cm}
\section{Mean-Field Game Model}
In this section, we derive the mean-field game attacker--defender model
and establish two-sided bounds for the defender density, ensuring
regular mass evolution.

\subsection{Model Derivation}\label{sec:MFG-deriv}
We denote by $\mathcal{P}_2(\mathbb{T}^d)$ the space of Borel probability
measures on $\mathbb{T}^d$ with finite second moment, and by $W_2$ the
Wasserstein-$2$ metric (see, for example,
\cite{villani2009wasserstein}).
As mentioned before, we employ this distance to quantify interactions
between populations; that is,
\vspace{-0.1cm}
$$R(\sigma_s,\sigma_t) = W_2(\sigma_s,\sigma_t).$$\vspace{-0.1cm}
Let us define
\vspace{-0.1cm}
\[
D(t,\nu_{H}) = (1-\alpha)\, d^{H}\!(W_2\big(\nu_{\mathrm{H}}, \mu(t)\big)),
\]\vspace{-0.1cm}
and set $E(t,m,\mu)=Q(t)$, where $Q(t)$ is given by \eqref{eq-Q-exp}.
To characterize the decentralized optimal strategy for the defender
population corresponding to \eqref{minimizer_revised-v}, we follow the
mean-field game framework (see
\cite{cardaliaguet2019master,carmona2018probabilistic} and references therein).
In particular, we isolate the macroscopic term appearing in the cost functional and denote it by
\[
\mathcal{F}[m]=D(t,\nu_{H})\cdot E(t,m,\mu).
\]
To capture how an individual agent influences the macroscopic cost, we
differentiate the macroscopic part of the cost, $\mathcal{F}[m]$, with respect to the measure $m$.
Since the Wasserstein space $\mathcal{P}_2(\mathbb{T}^d)$ lacks a linear
structure, we employ the Lions derivative. 
Note that 
\begin{equation}\label{eq:Fdem}
    \frac{\delta \mathcal{F}}{\delta m}[m](x)
= D(t,\nu_{H})\,  \frac{\delta}{\delta m}E(t, m,\mu),
\end{equation}
and recalling \eqref{eq-Q-exp}, we have
\begin{equation*}
    \begin{split}
         &\frac{\delta}{\delta m}E(t, m,\mu) =  -E(t, m,\mu)\times \\&\times (d^{att})'(W_2(\mu(s), m(s)))\cdot\frac{\delta}{\delta m}[W_2^2(\mu, m)].
    \end{split}
\end{equation*}
Hence, the functional $\mathcal{F}[m]$ is Lions differentiable.

For a given population flow $m(t)$ and a point $x\in\mathbb T^d$,
we perturb the defender distribution $m(t)$ by inserting an infinitesimal
amount of mass at $x$. More precisely, for $\varepsilon\in(0,1)$, we set
\[
m^\varepsilon (t):= (1-\varepsilon)m(t)+\varepsilon\delta_x,
\]
where $\delta_x$ denotes the Dirac mass at $x$. Since $m^\varepsilon(t)$ is a
convex combination of two probability measures, we have
$m^\varepsilon(t)\in\mathcal P(\mathbb T^d)$.
By the first-order expansion of $\mathcal F$ with respect to the measure
variable, we have
\[
\mathcal F[m^\varepsilon(t)]-\mathcal F[m(t)]
=
\varepsilon\left(
\frac{\delta \mathcal F}{\delta m}[m(t)]-c(t)
\right)+o(\varepsilon),
\]
where
\[
c(t):=\int_{\mathbb T^d}
\frac{\delta \mathcal F}{\delta m}[m(t)](y)\,m(t)(dy).
\]
At the macroscopic level, $c(t)$ depends on the population flow and 
implicitly on the control. However, in the representative-agent setting (mean-field game framework),
the population flow $m(t)$ is frozen and treated as prescribed.
Therefore, $c(t)$ becomes a time-dependent term independent of the state
and control of the individual agent, and so it does not affect the
minimizer.

Thus, we have 
the following minimization problem for representative defender at $x$:
\begin{equation}
\label{eq:rep_agent_problem}
\inf_{u}\;
\mathbb E\!\left[
\int_0^T
L(t,X(t),\nu_{H}, \mu, m, u)
\,dt
\right],
\end{equation}
subject to \eqref{McKean-Vlasov}, where
\begin{equation*}
    L(t, x,\nu_{H}, \mu, m, u) =D(t,\nu_{H})\cdot\frac{\delta}{\delta m}E(t, m,\mu)+\alpha u^2.
\end{equation*}
The  corresponding Hamiltonian is given by
\begin{multline}\label{mfg: Hamilton}
        H(t,x,\nu_{H}, \mu, m, p) =\underset{q}{max}(-L( t,x, \nu_{H}, \mu, m, q) - q\cdot p)\\=\frac{p^2}{4\alpha}-D(t,\nu_{H})\cdot\frac{\delta}{\delta m}E(t, m,\mu).
\end{multline}
Denoting
\begin{equation*}
     F(t,x, \nu_{H}, \mu, m) =\frac{\delta \mathcal{F}}{\delta m}[m](x),
\end{equation*}
we derive the system
 of MFG attacker-defender model.
\begin{problem}\label{prob-MFG}
Given the distribution of HVU population, $\nu_{H}$, and the state evolution of the attacker distribution $\mu(t)$,  find $(w,m)$ such that it solves the following coupled system of PDE's 
\begin{equation}\label{eq:MFG}
       \begin{cases}
        -w_t-\varepsilon\Delta w + \frac{|Dw|^2}{4\alpha}=F(t,x, \nu_{H},\mu, m)\\
         m_t -\varepsilon\Delta m-\frac{1}{2\alpha}\div_x(Dw\cdot m)=0\\
        \mu_t - \varepsilon\Delta \mu -\div_x(u_{att}\cdot \mu)=0,\\
         w(T,x) = 0,\; m(0,x)=m_0, \; x\in\T^d,
\end{cases}
\end{equation}
for $(t,x)\in(0,T)\times \T^d.$
\end{problem}

\subsection{Flow Bounds}\label{sec:flowb}
This section is devoted to the qualitative analysis of the MFG model. In
particular, we establish uniform upper and lower bounds for the density
of the defender population, ensuring regular mass evolution over time.
\begin{theorem}
\label{thm:fp_bounds}
Let $T > 0$ and let $(w,m)$ be a classical solution to Problem \ref{prob-MFG},
with initial condition $m(0, \cdot) = m_0$ such that $0 < \min_{\mathbb{T}^d} m_0 \le \max_{\mathbb{T}^d} m_0 < \infty$.
Then, for all $(t,x) \in [0,T] \times \mathbb{T}^d$, the solution $m$ satisfies the following bounds  
\begin{equation*}
e^{-Kt}\min_{\mathbb{T}^d} m_0\le m(t,x) \le e^{Kt}\max_{\mathbb{T}^d} m_0 ,
\end{equation*}
where 
$K= \frac{1}{2\alpha}\underset{t\in[0,T]}{\sup}\|\Delta u\|_{L^\infty(\T^d)}.$
\end{theorem}
\begin{proof}
Since $w\in C^{1,2}([0,T]\times \T^d)$, then 
 $\Delta w$ is continuous. Hence, there exists a positive constant $K$ such that 
 \begin{equation*}
     \frac{1}{2\alpha} \|\Delta w\|_{L^\infty([0,T]\times\T^d)}\le K.
 \end{equation*}
 
Because $m \in C^{1,2}([0,T]\times\mathbb{T}^d)$ is a classical solution
to the Fokker--Planck equation (the second equation in \eqref{eq:MFG}),
for each $t\in[0,T]$ the function $x \mapsto m(t,x)$ attains its maximum
at some point $x_t^M \in \mathbb{T}^d$. Let $L(t) = \max_{x \in \mathbb{T}^d} m(t,x)=m(t,x_t^M)$. Since the system is defined  on the boundary-free torus, then,  we have that all  points are strictly interior. Therefore,  $\nabla m(t, x^M_t) = 0$ and $-\Delta m(t, x^M_t) \ge 0$.

The Fokker-Planck equation can be rewritten in non-divergence form as follows
\begin{equation}\label{eq:used}
       m_t - \varepsilon\Delta m + \frac{1}{2\alpha}\nabla w \cdot \nabla m = - \frac{1}{2\alpha}m\Delta w .
\end{equation}

Evaluating the equation at the maximum point $x_t^M$, and using the
previous observations, we obtain
\begin{align*}
    \dot{L}(t) \le - \frac{1}{2\alpha}(\Delta w)(t, x_t^M)\cdot L(t) 
    \le  \frac{1}{2\alpha}\|\Delta w\|_\infty\cdot L(t) 
    \le K\cdot L(t).
\end{align*}
By Grönwall's inequality, we obtain the upper bound $$L(t) \le L(0)e^{Kt}.$$

Similarly, letting $l(t) = \min_{x \in \mathbb{T}^d} m(t,x)$ at a minimum point $x^m_t$, we have $\nabla m(t, x^m_t) = 0$ and $-\Delta m(t, x^m_t) \le 0$. Along with the FP equation in divergence form \eqref{eq:used}, we get
\begin{align*}
    \dot{l}(t)\ge - \frac{1}{2\alpha}(\Delta w)(t, x^m_t)\cdot l(t) 
    \ge - \frac{1}{2\alpha}\|\Delta w\|_\infty\cdot l(t) 
    \ge -K\cdot l(t).
\end{align*}
Again by Grönwall's inequality, we get the lower bound $$l(t) \ge l(0)\cdot e^{-Kt}.$$
\end{proof}

\addtolength{\textheight}{-3cm}   
 
\section{Simulations}\label{sec:sim}
In this section, we describe the numerical approach used to solve the MFG
attacker--defender system in \eqref{eq:MFG}, and present illustrative
scenarios highlighting the ability of the model to capture key emergent
behaviors.

\begin{figure*}[!htbp]

\centering
\vspace{-0.2cm} 

\subfloat{%
\includegraphics[width=0.33\textwidth]{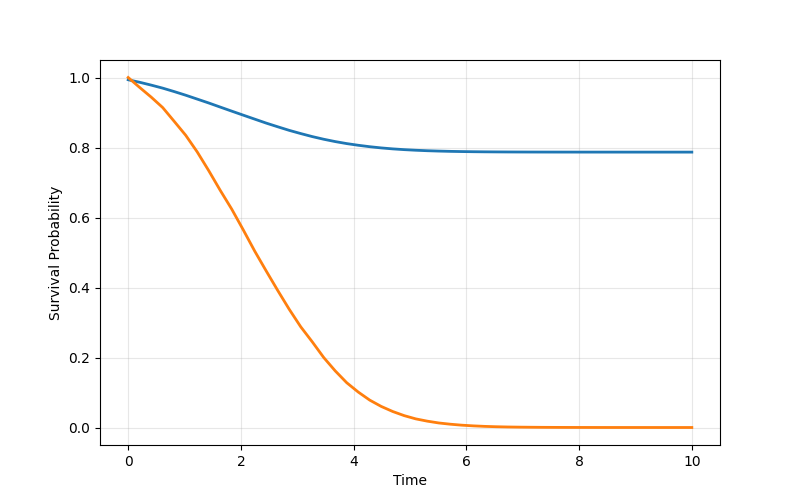}}
\hfill
{%
\includegraphics[width=0.33\textwidth]{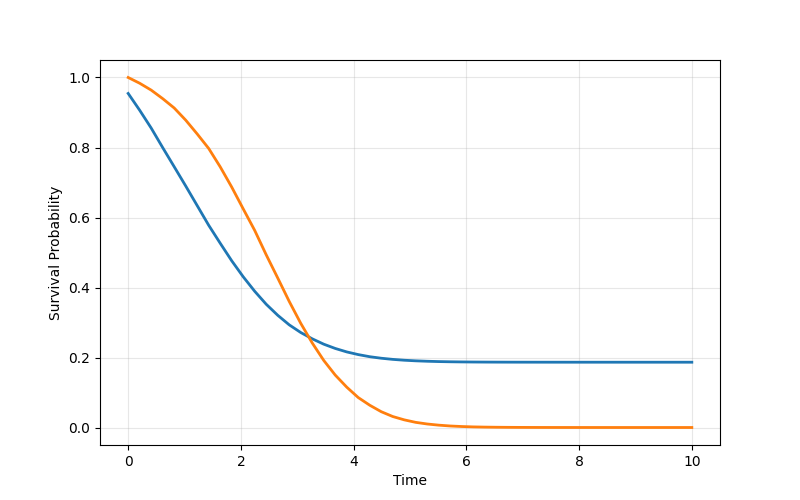}}
\hfill
{%
\includegraphics[width=0.33\textwidth]{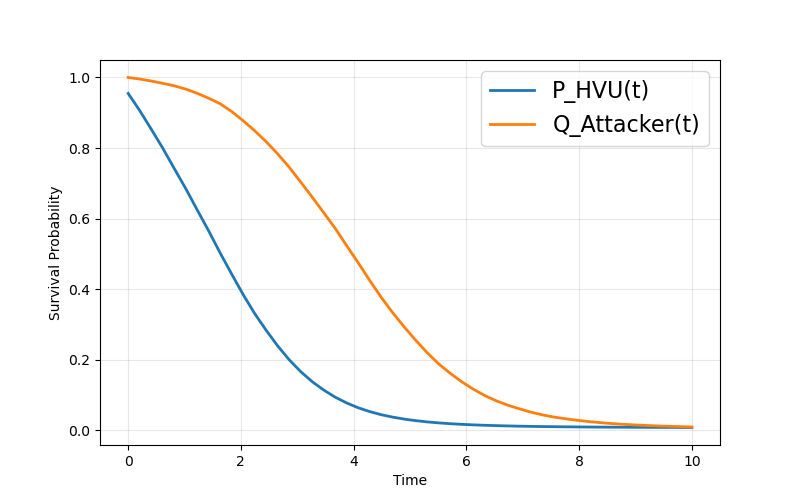}}

\caption{The blue curve represents the survival probability of the HVU, while the orange curve corresponds to the attackers.The initial position of the defenders is $(-3,-3)$  with variance $\sigma^2 = 0.85$.
The attackers start from $(-4,4)$ with variance $\sigma^2 = 0.85$ and move toward $(1,1)$, where the HVU is located with variance $\sigma^2 = 0.1$. For the first panel, the attrition rate parameters are $\sigma_A = \sigma_H = 5,\ \lambda_A = 14,\ \lambda_H = 1$. 
For the second panel, $\sigma_A = \sigma_H = 5,\ \lambda_A =\lambda_H = 7$. 
For the third panel, $\sigma_A = \sigma_H = 5,\ \lambda_A = 2,\ \lambda_H = 7$.}
\label{surv_won_lose}
\vspace{-0.5cm}
\end{figure*}

\begin{figure*}[!htbp]
\centering
\subfloat{%
\includegraphics[width=0.33\textwidth]{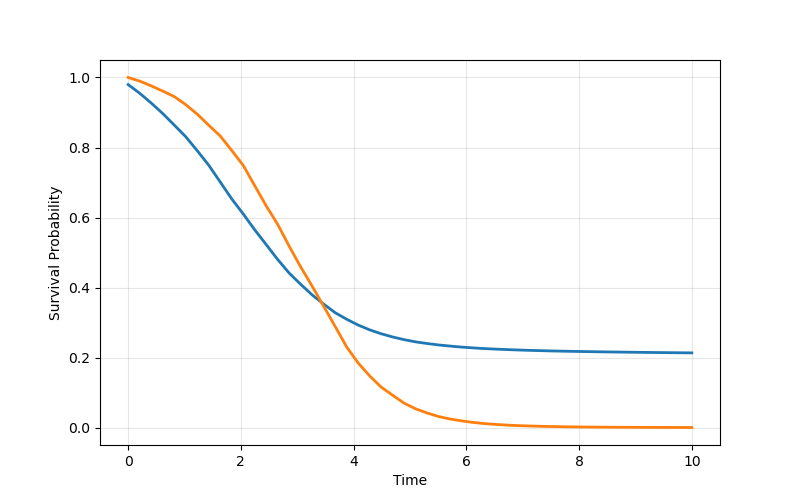}}
\hfill
\subfloat{%
\includegraphics[width=0.33\textwidth]{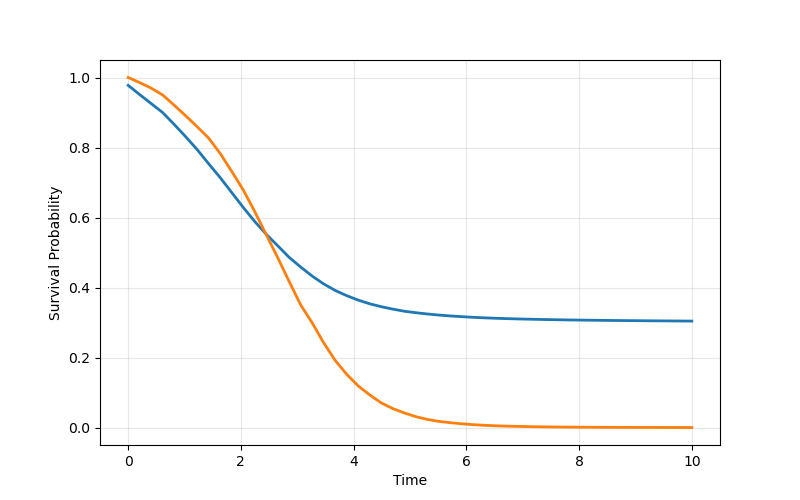}}
\hfill
{%
\includegraphics[width=0.33\textwidth]{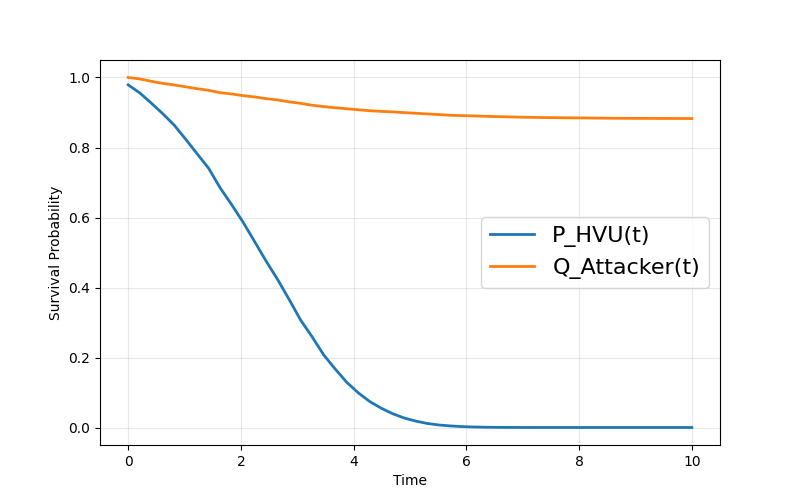}}
\caption{Comparison of survival probabilities for three values of the variance of
the initial defender distribution $m_0$:
$\sigma^2=0.35,\,1.4,\,$ and $1.8$. The defenders are initially centered
at $(0,0)$, while the attackers start from $(-4,0)$ with variance
$\sigma^2=1.5$ and move toward $(-4,4)$, where the HVU is located with
variance $\sigma^2=0.2$. The attrition rate parameters are
$\sigma_A=\sigma_H=2$ and $\lambda_A=3$, $\lambda_H=10$.}
\label{surv_variances}
\vspace{-0.6cm}
\end{figure*}
\vspace{-0.4cm}
\subsection{Numerical Method and Simulation Setup}

A key challenge in numerically solving the proposed MFG system lies in
the computation of the Wasserstein-$2$ distance, which is
computationally demanding and must be evaluated at each time step. To
mitigate this difficulty, we approximate the Wasserstein-$2$ distance
using the Sinkhorn algorithm, see \cite{cuturi2013sinkhorn}. The resulting problem is then solved using
physics-informed neural networks (PINNs), see \cite{pinn}.

We consider a two-dimensional setting, with spatial domain
$\Omega = [-5,5]^2$ and time horizon $T=10$.  The spatial domain is discretized using a uniform
$60 \times 60$ grid, while the time interval is sampled uniformly with
$50$ points. These collocation points are used to train the PINN model. In our model, we set the
diffusion coefficient to $\varepsilon = 0.001$ and the Pareto weight to
$\alpha = 0.1$.
As $d^{att}$ and $d^{H}$ functions, we take $$d^{att}\left(\theta_1, \theta_2\right)=\lambda_A\cdot \exp\left(-\tfrac{W_2^2(\theta_1, \theta_2)}{\sigma_A}\right)$$ 
and 
$$d^{H}(\theta_1, \theta_2)=\lambda_H\cdot \exp\left(-\tfrac{W_2^2(\theta_1, \theta_2)}{\sigma_H}\right).$$

 As mentioned before, to numerically solve the mean-field game system,
we first approximate the MFG system \eqref{eq:MFG} by replacing the
Wasserstein-$2$ distance with its Sinkhorn approximation. We then solve
the resulting approximated problem using PINNs. In this approach, the loss function is built from the residuals of the coupled Hamilton–Jacobi and Fokker–Planck equations. It also includes the terminal and initial conditions of the problem.
We use two separate neural networks: one for the value function 
$w(t,x,y)$ and one for the density $m(t,x,y)$. Both networks have the same ResNet architecture. It consists of an input layer with three neurons, followed by a fully connected hidden layer with 128 neurons and Tanh activation function. This is followed by four residual blocks, each composed of two fully connected layers with 128 neurons, Tanh activation functions, and a residual connection. The network ends with a final linear output layer with one neuron. For the density network, we apply Softplus activation function. 
We do not use labeled data for training. Instead, the networks learn by minimizing the equation errors. Trainings were conducted for 15,000 epochs with the Adam optimizer, utilizing NVIDIA GeForce RTX 4080 and 5080 GPUs.

\vspace{-0.2cm}
\subsection{Simulation Scenarios}
We consider two classes of simulations. First, by varying the parameters
of the attrition rate functions, we demonstrate the model’s ability to
capture different weapon strengths and their impact on the outcome.
Second, we perform simulations with varying initial spreads of the
defender population, illustrating how population dispersion influences
the effectiveness of fire and the resulting outcomes.

Fig.~\ref{surv_won_lose} illustrates the impact of the attrition rate
parameters on the survival probabilities of the HVU and the attackers.
In all three scenarios (first, second, and third panels), the defenders,
attackers, and HVU share the same initial positions and variances, and
the attacker dynamics are identical. The only difference lies in the
attrition rate parameters. 
In the first panel, the defenders are stronger than the attackers; in
the second, both sides have equal strength; and in the third, the
attackers are stronger than the defenders. This results in high,
moderate, and low survival probabilities for the HVU population, respectively, and
conversely for the attackers.

\begin{figure*}[htbp!]
\vspace{-0.5cm}

\centering

\subfloat[\centering Top-Down View $t = 0.00$]{%
\includegraphics[width=0.25\textwidth]{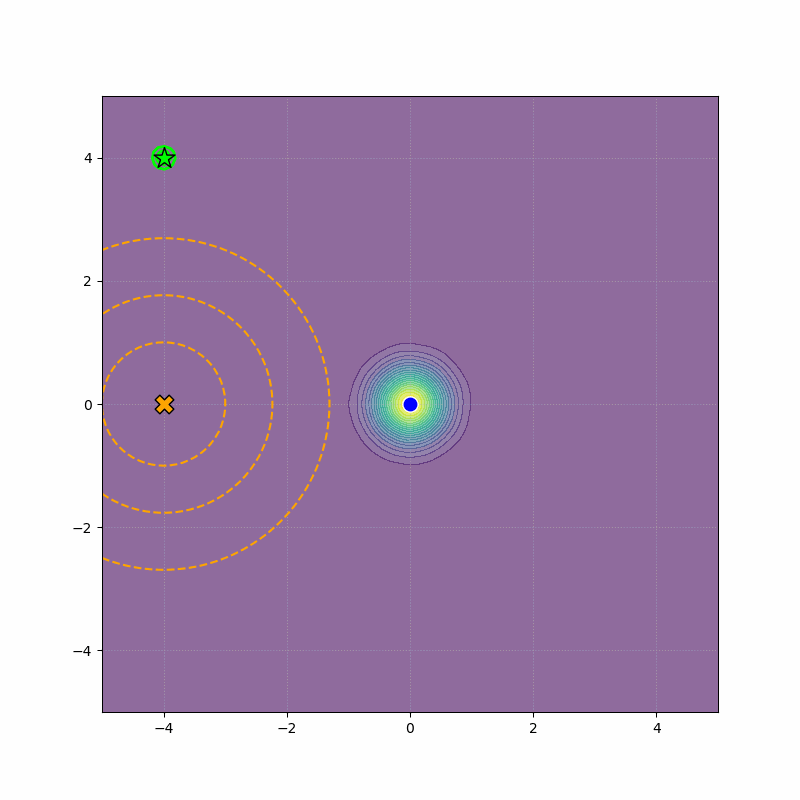}}
\hfill
\subfloat[\centering Top-Down View $t = 3.27$]{%
\includegraphics[width=0.25\textwidth]{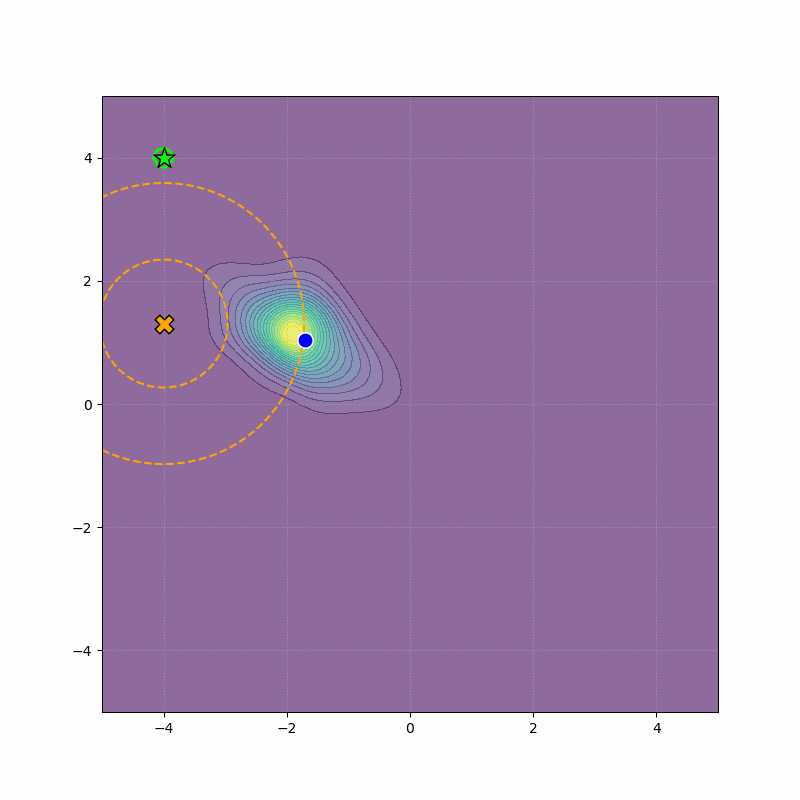}}
\hfill
\subfloat[\centering Top-Down View $t = 6.53$]{%
\includegraphics[width=0.25\textwidth]{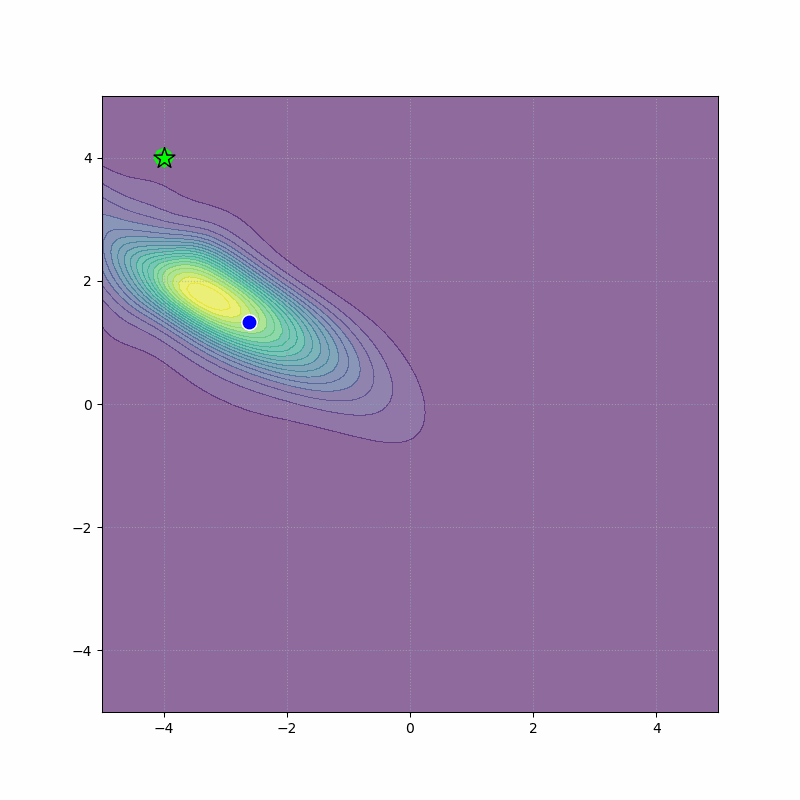}}
\hfill
\subfloat[\centering Top-Down View $t = 10.00$]{%
\includegraphics[width=0.25\textwidth]{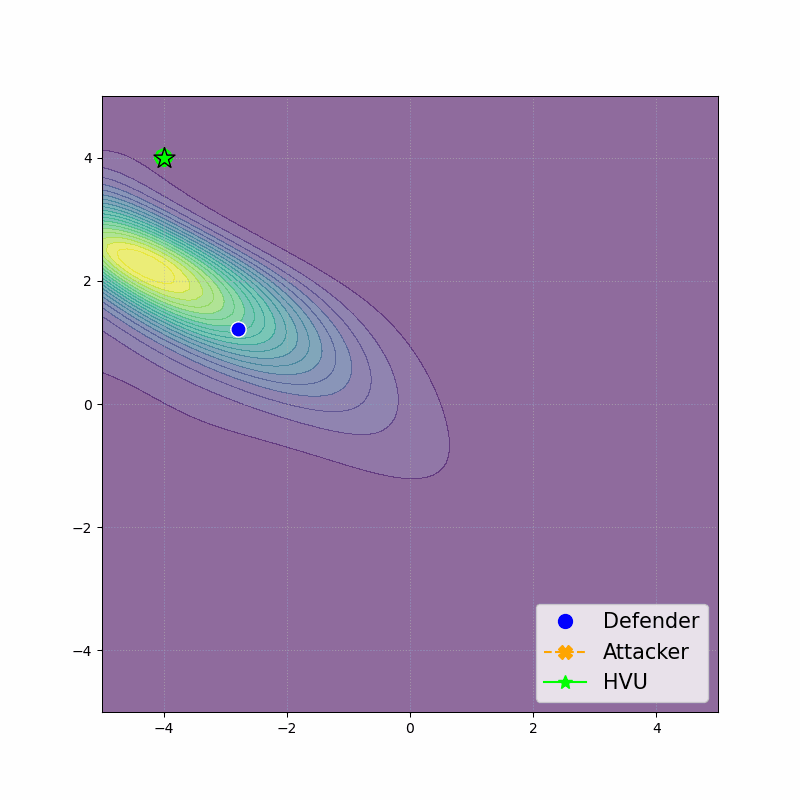}}

\caption{Top-down view of the evolution of the attacker and defender over time of first scenario ($\sigma^2 = 0.35$) of Fig. \ref{surv_variances}.
}
\vspace{-0.4cm}
\label{0.35}
\end{figure*}
\begin{figure*}[!htbp]
\centering
\subfloat[\centering Top-Down View $t = 0.00$]{%
\includegraphics[width=0.25\textwidth]{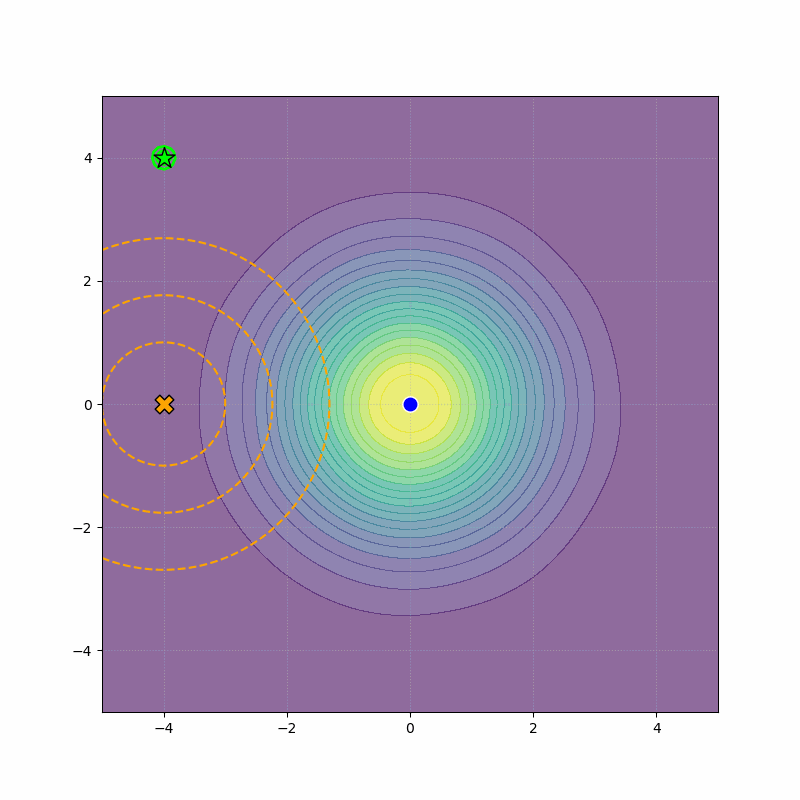}}
\hfill
\subfloat[\centering Top-Down View $t = 3.27$]{%
\includegraphics[width=0.25\textwidth]{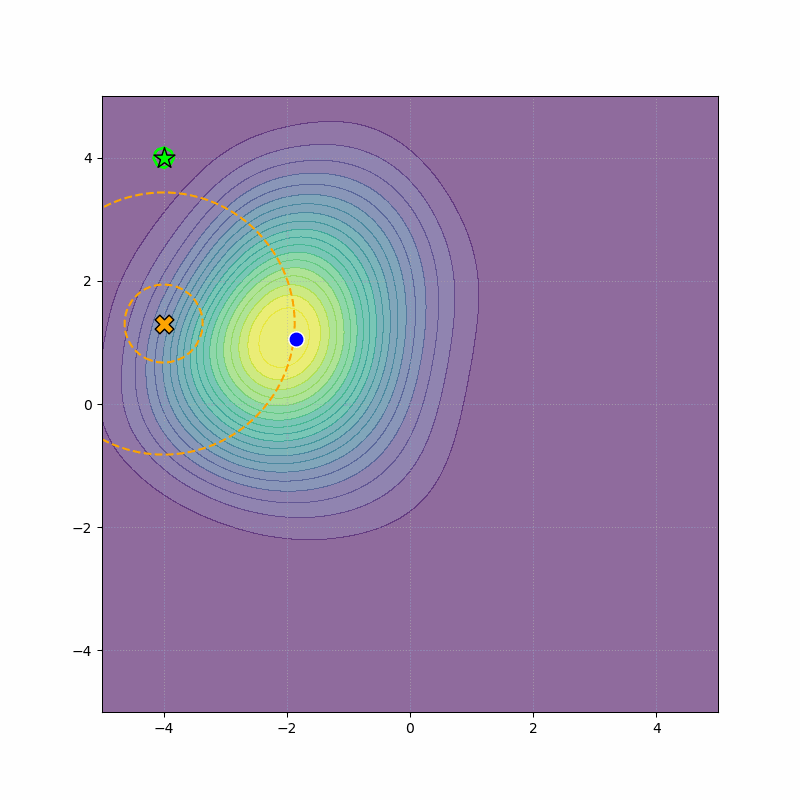}}
\hfill
\subfloat[\centering Top-Down View $t = 6.53$]{%
\includegraphics[width=0.25\textwidth]{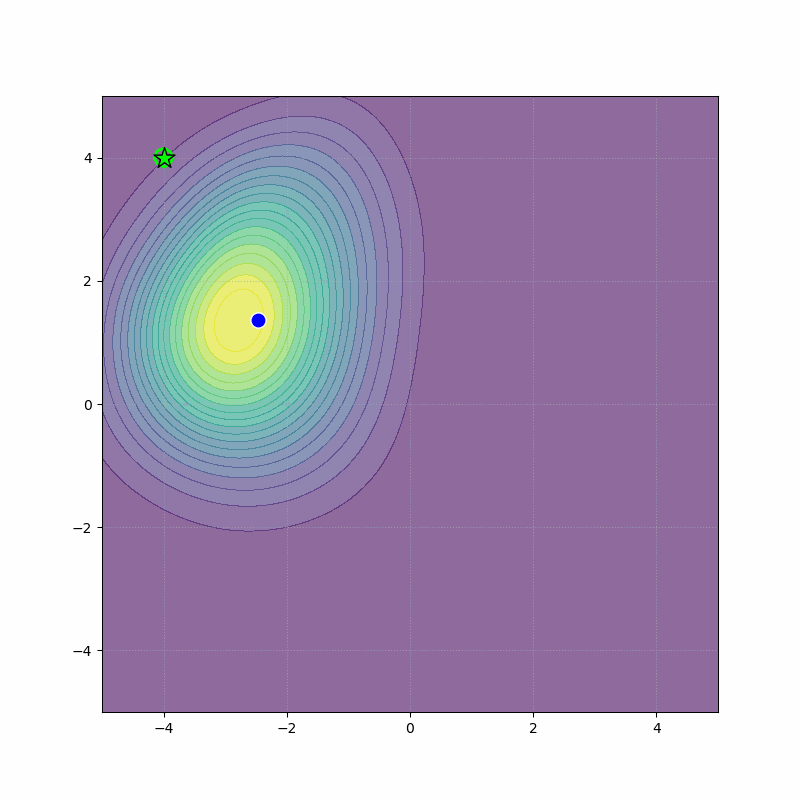}}
\hfill
\subfloat[\centering Top-Down View $t = 10.00$]{%
\includegraphics[width=0.25\textwidth]{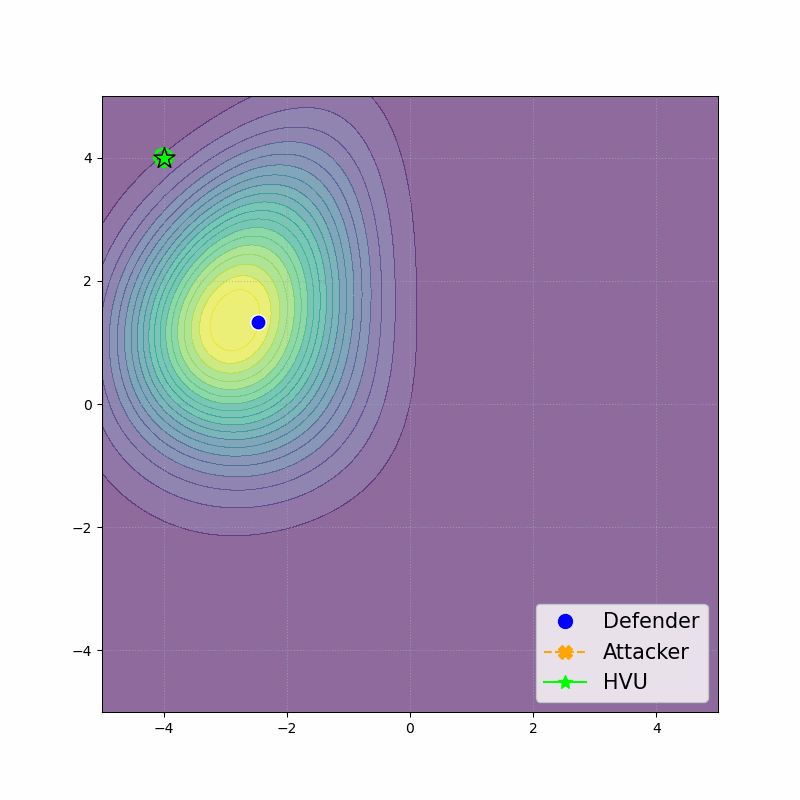}}
\caption{Top-down view of the evolution of the attacker and defender over time
for the second scenario ($\sigma^2 = 1.4$) in Fig.~\ref{surv_variances}.}
\label{1.4}
\end{figure*}
\begin{figure*}[!htbp]
\vspace{-0.4cm}
\centering

\subfloat[\centering Top-Down View $t = 0.00$]{%
\includegraphics[width=0.25\textwidth]{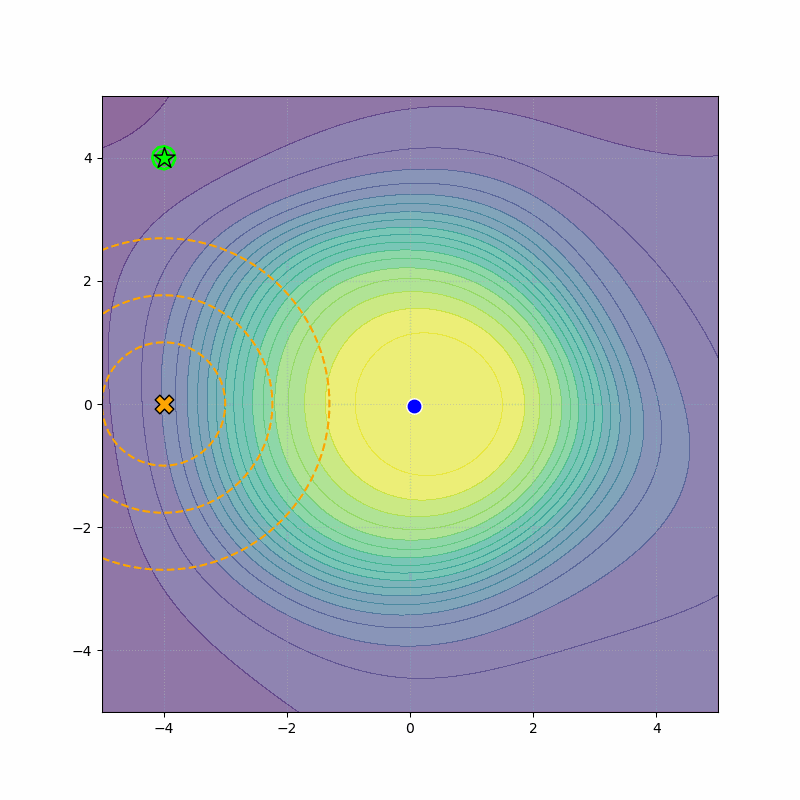}}
\hfill
\subfloat[\centering Top-Down View $t = 3.27$]{%
\includegraphics[width=0.25\textwidth]{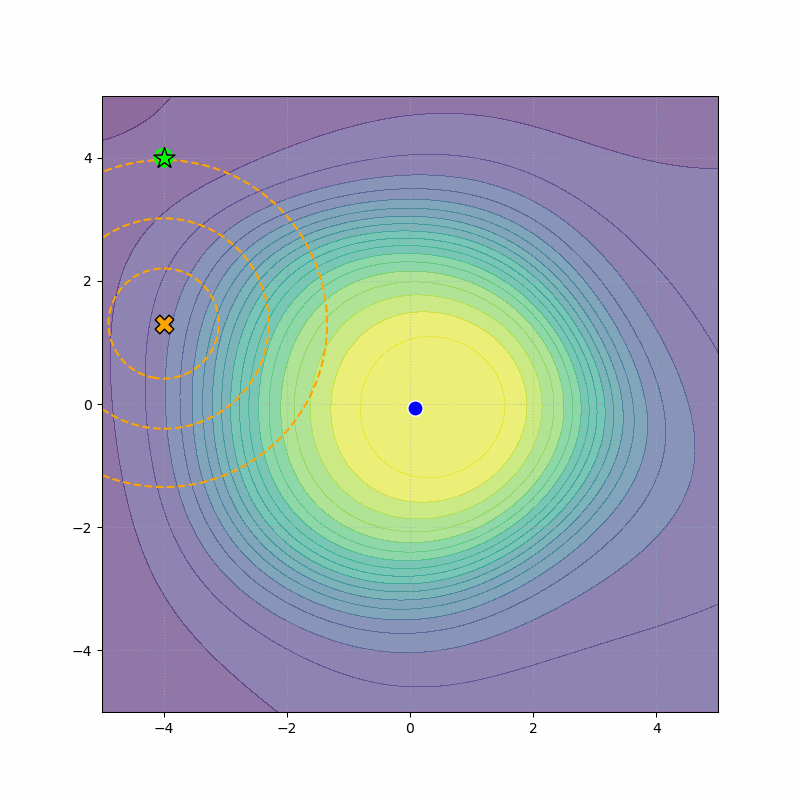}}
\hfill
\subfloat[\centering Top-Down View $t = 6.53$]{%
\includegraphics[width=0.25\textwidth]{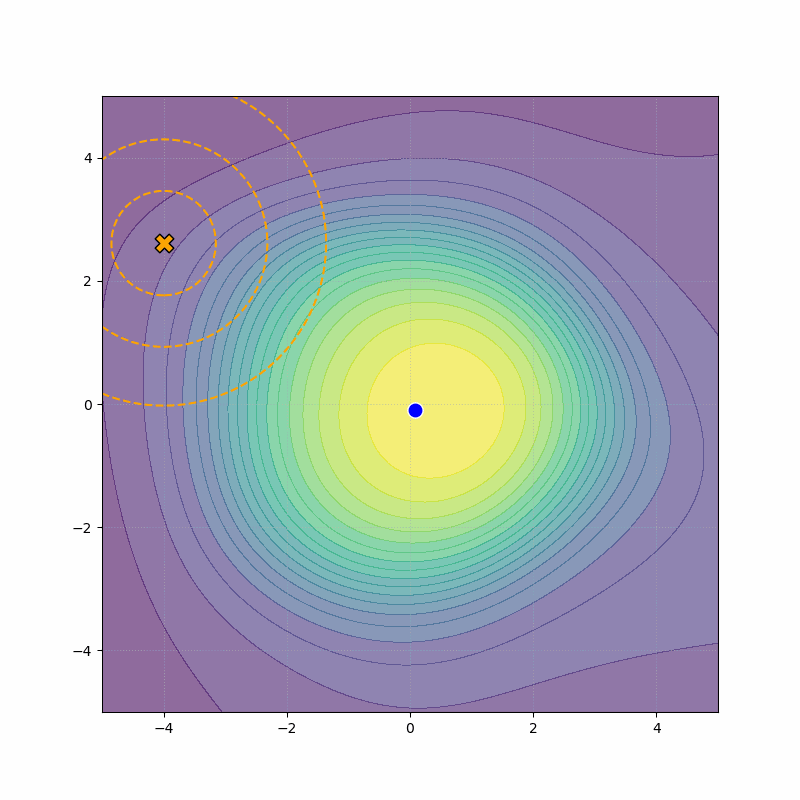}}
\hfill
\subfloat[\centering Top-Down View $t = 10.00$]{%
\includegraphics[width=0.25\textwidth]{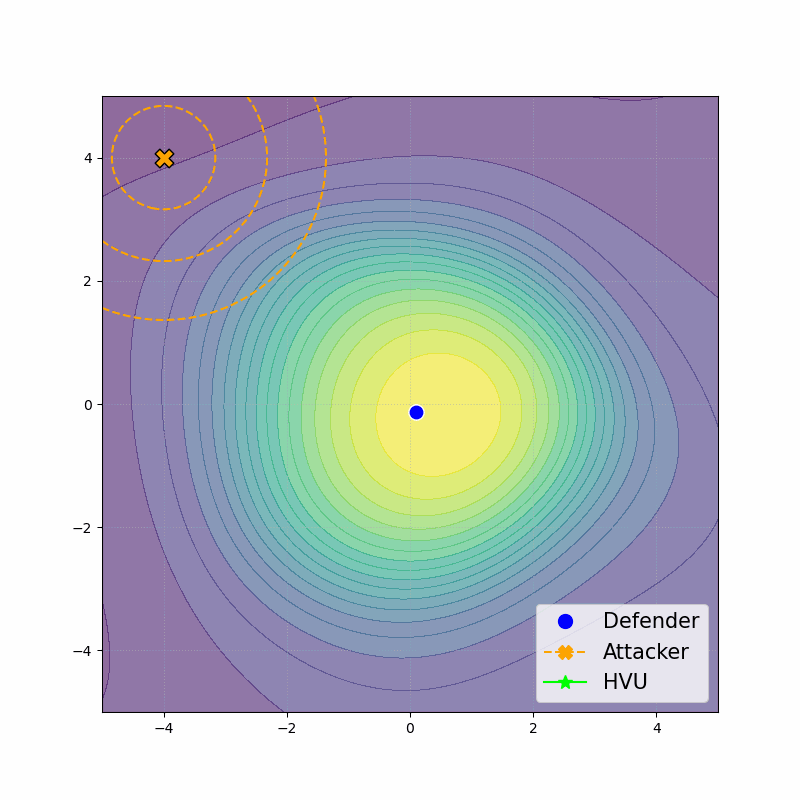}}

\caption{Top-down view of the evolution of the attacker and defender over time
for the third scenario ($\sigma^2 = 1.8$) in Fig.~\ref{surv_variances}.}
\label{2.5}
\end{figure*}

Figures \ref{0.35}, \ref{1.4}, and \ref{2.5} show the evolution of the attacker and defender populations over time for different values of the defender distribution variance. Fig.~\ref{surv_variances} shows the corresponding survival probabilities of the HVU and the attackers.
These results illustrate the impact of the defenders initial distribution variance on the model outcome.

According to  Fig. \ref{fig:distances}, due to smallness of  the statistical distance between the attacker and defenders, the defenders are most effective when their distribution is moderately spread $(\sigma^2 = 1.4)$, as shown in Fig.~\ref{1.4}. This behavior is also reflected in the survival probability plots in Fig.~\ref{surv_variances} (second panel).
For the case with small variance $(\sigma^2 = 0.35)$, the defenders start in a highly concentrated distribution. However, during the simulation they gradually spread and move toward a more favorable distribution, as shown in Fig.~\ref{0.35}. As a result, the HVU still survives, although with a lower probability, as seen in Fig.~\ref{surv_variances} (first panel), since the defenders are not optimally distributed during the entire time interval.
In contrast, when the initial variance is large $(\sigma^2 = 2.5)$, the defenders start too dispersed, as shown in Fig.~\ref{2.5}. In this case, the simulation time of 10 seconds is not enough for the density to reorganize and become concentrated enough. As a result, the HVU is destroyed, as seen in Fig.~\ref{surv_variances} (third panel).

The different behaviors in Fig.~\ref{surv_variances} can be interpreted
in light of the theoretical bounds in Theorem~\ref{thm:fp_bounds}. By
Theorem~\ref{thm:fp_bounds}, the density evolution is controlled by the
minimum and maximum values of $m_0$. Hence, depending on the spread of
$m_0$, the defenders require different amounts of time to redistribute
toward a more effective configuration.
Particularly, in the moderately
spread case ($\sigma^2 = 1.4$), these values remain nearly unchanged and
stay close to $4.68 \times 10^{-6}$ and $8.96 \times 10^{-4}$,
respectively. In the more concentrated case
($\sigma^2 = 0.35$), the minimum and maximum values approach those of
the moderately spread case only after approximately $T=5$ and remain
around these levels thereafter. As a result, the survival probability of
the HVU decreases, but remains positive.

\vspace{-0.3cm}
\section{Conclusion and Future Work}

We introduced a mean-field game framework for large-scale
attacker--defender systems with the objective of protecting one or
multiple high-value units. By replacing agent-wise attrition with a
population-wise interaction mechanism based on the Wasserstein-$2$
distance, we derived a macroscopic attacker--defender model in the form
of a coupled HJB--FP system. We also established two-sided bounds for
the defender density, ensuring regular mass evolution, and proposed a
numerical scheme based on Sinkhorn approximation and PINNs. The
simulation results demonstrated that the model captures meaningful
strategic effects, including the impact of weapon strength and defender
dispersion on mission outcome.

Future work will focus on establishing a more complete theoretical
foundation for the proposed system, including existence and uniqueness
results, as well as extending the framework to more general attrition
mechanisms, heterogeneous populations, and moving HVUs.
\vspace{-0.2cm}





\bibliographystyle{acm}
\bibliography{cdcBib}

\end{document}